\documentclass[11pt]{article}
\usepackage{amsmath,amsfonts,amsthm,amscd,amssymb,graphicx}
\usepackage{subfigure}

\numberwithin{equation}{section}

\usepackage{hyperref}
\usepackage{verbatim} 

\newtheorem{theorem}{Theorem}[section]

\newtheorem{lemma}[theorem]{Lemma}

\newtheorem{proposition}[theorem]{Proposition}

\newtheorem{remark}[theorem]{Remark}

\def\eps{\varepsilon }

\newcommand{\RR}{\mathbb{R}}

\newcommand{\CC}{\mathbb{C}}

\newcommand{\TT}{{\mathbb T}}

\def\beq{\begin{equation}}
\def\eeq{\end{equation}}
\def\bb1{{1\!\!1}}
\def\R{\mbox{Re }}

\def\w{{\omega}}

\def\eps{\varepsilon}

\def\triangle{\Delta}
\def\bega{\begin{aligned}}
\def\enda{\end{aligned}}
\def\pt{\partial}
\def\w{\omega}
\def\lw{\left}
\def\rw{\right}
\begin{document}

\title{The inviscid limit of Navier-Stokes with critical Navier-slip boundary conditions for analytic data } 

\author{
 Trinh T. Nguyen\footnotemark[1]
}

\maketitle

\renewcommand{\thefootnote}{\fnsymbol{footnote}}

\footnotetext[1]{Department of Mathematics, Penn State University, State College, PA 16803. Emails: 
txn5114@psu.edu.}



\begin{abstract}
In this paper, we establish the short time inviscid limit of the incompressible Navier-Stokes equations with critical Navier-slip boundary conditions for analytic data on half-space, a boundary condition that is physically derived from the hydrodynamic limit of the Boltzmann equations with the Maxwell boundary conditions. The analysis is built upon the recent framework developed by T. T. Nguyen and T. T. Nguyen (Arch. Ration. Mech. Anal., 230(3):1103-1129, 2018.) in the case of the classical no-slip boundary conditions. The novelty in this paper is to derive the precise pointwise bound on the Green kernel for the Stokes problem with a nonlocal boundary condition and to propagate the boundary layer behavior for vorticity. 

\end{abstract}



\section{Introduction}
In this paper, we are interested in the inviscid limit of the Navier-Stokes equations for incompressible fluids
\beq \label{NS}
\begin{aligned} 
\partial_{t}u+u\cdot\nabla u+\nabla p&=\nu\triangle u,
\\
\nabla\cdot u&=0
\end{aligned}
\eeq 
posed on the half space $(x,y)\in \TT \times \RR_+$, with the slip boundary condition
\beq \label{NS-BC}
\begin{aligned}
u_2=0\qquad \text{and}\qquad \pt_y u_1=\nu^{-\beta}u_1\qquad\text{when}\quad y=0.
\end{aligned}
\eeq 
Here $\nu^{\beta}$ is the slip length and $u=u(t,x,y)=(u_1(t,x,y),u_2(t,x,y))\in \mathbb{R}^2$ is the velocity field. 
The goal of this paper is to justify the inviscid limit for analytic data in the critical case $\beta=1$.

First, let us mention some previous works on the inviscid limit and boundary layer theory for $\beta\in [0,1)$. For $\beta=0$, in which the slip length does not depend on $\nu$, the picture is now complete: \cite{Sueur} derives a complete boundary layers expansion, and \cite{MasRou1}  justifies the vanishing viscosity limit by a compactness argument for any bounded domain and for half-space (see also \cite{Bardos,Robert,Pla,Kell}). For $0\le \beta<1$, the inviscid limit is established in \cite{Paddick} with a rate of convergence $O\lw(\nu^{\frac{1-\beta}{2}}\rw)$ for Sobolev data, while the boundary layer expansion is proved to fail when $\beta=\frac{1}{2}$. When $\beta=1$, a Kato-type criterion for the inviscid limit to hold is proved in \cite{kato}.

\subsection{Criticality of $\beta=1$}
When $\beta=1$, we have the \textit{critical-slip} boundary condition 
\beq \label{cri}
\pt_y u_1|_{y=0}=\nu^{-1}u_1|_{y=0}.
\eeq 
The boundary condition \eqref{cri} is physically obtained from the hydrodynamic limit of the Boltzmann equations with the Maxwell boundary conditions (see \cite{Masmoudi-critical}).  
 However, the inviscid limit for the critical case $\beta=1$ and stability of boundary layer expansions for \eqref{cri} remain open, due to the failure of standard energy estimates and even the lack of approriate boundary layer theory for the Navier-Stokes equations with the boundary condition \eqref{cri}. Our work appears to be the first giving an affirmative answer to the inviscid limit problem in the critical case $\beta=1$, with the assumption only placed on the initial data. The inviscid limit holds uniformly in a short time interval independent of the viscosity $\nu$. There are numerical evidences that the inviscid limit fails for longer time, leading to anomalous dissipation of the Navier-Stokes equations (see \cite{Yen}). At the time when the inviscid limit may fail, 
 there would be possible emergence of weak solutions to the Euler equations in the vanishing viscosity limit (see \cite{TheoHuy}).
  
To illustrate why $\beta=1$ is considered to be critical, let us give a proof of the following theorem for any $\beta\in[0,1)$ for the reader convenience. The proof is quite simple and is originally done in \cite{Paddick}.  

 \begin{theorem}\label{thm}(\cite{Paddick})
 Let $u^E\in C^1\lw((0,T),W^{2,\infty}(\Omega)\cap L^2(\Omega)\rw)$ be a smooth solution to Euler with the non-penetration boundary condition $u^E_2|_{y=0}=0$ and $u^\nu$ be the solution to the Navier-Stokes equations with the slip boundary condition \eqref{NS-BC} with $0\le \beta<1$ on the domain $\Omega=\mathbb{T}\times \mathbb{R}_+$. Then 
 \[
 \sup_{0\le t\le T}\|u^\nu(t)-u^E(t)\|_{L^2(\Omega)}\le C_T \nu^{\frac{1-\beta}{2}}+\|u^\nu(0)-u^E(0)\|_{L^2(\Omega)}.
 \]
 The convergence holds for any finite time $T>0$, which is the time of existence of Euler solutions in Sobolev space.
 \end{theorem}
\begin{proof}
Let $v=u^\nu-u^E$ be the difference between the velocity of Navier-Stokes equations and Euler equations. Then $v$ solves 
 \[\bega
 \pt_t v+u^E\cdot \nabla v+v\cdot \nabla u^E+v\cdot \nabla v-\nu\triangle u^\nu&=-\nabla (p^\nu-p^E).\\
\enda 
 \]
 Multiplying both sides of the first equation by $v$, integrating over $\Omega=\mathbb{T}\times \mathbb{R}_+$ and using the non-penetration boundary condition, we have: 
 \beq\label{Sva1}
 \bega 
 &\frac{1}{2}\frac{d}{dt}\|v\|_{L^2}^2+\int_{\Omega}(v\cdot\nabla u^E)\cdot v-\nu\int_\Omega \triangle u^\nu\cdot v=0.\\
 \enda
 \eeq 
 By integrating by parts and the slip boundary conditions, we have 
 \[\bega
 -\nu\int_{\Omega}\triangle u^\nu\cdot v&
 =\nu\int_{\Omega}|\nabla v|^2+\nu\int_{\Omega}\nabla u^E\cdot \nabla v+\nu^{1-\beta}\int_{\mathbb{T}}|u_1^\nu(t,x,0)|^2dx-\nu^{1-\beta}\int_{\mathbb{T}}u_1^\nu(t,x,0)u_1^E(t,x,0)dx
 \enda 
 \]
 Combining the above with \eqref{Sva1}, we have 
 \beq\label{Sva2}
 \bega
& \frac{1}{2}\frac{d}{dt}\|v\|_{L^2}^2+\nu\int_{\Omega}|\nabla v|^2+\nu^{1-\beta}\int_{\mathbb{T}}u_1^\nu(t,x,0)^2dx=\int_{\Omega}(v\cdot \nabla u^E)\cdot v\\
 &-\nu\int_{\Omega}\nabla u^E\cdot \nabla v+\nu^{1-\beta}\int_{\mathbb T}u_1^\nu(t,x,0)u_1^E(t,x,0)dx.
 \enda 
 \eeq
By a standard Cauchy inequality $ab\le \eps a^2+C_\eps b^2$ for $\eps>0$ small, we get 
 \[
 \frac{1}{2}\frac{d}{dt}\|v\|_{L^2}^2\le \|\nabla u^E\|_{L^\infty}\|v\|_{L^2}^2+C_\eps \nu\|\nabla u^E\|_{L^2}^2+C_\eps \nu^{1-\beta}\int_{\mathbb T}|u_1^E(t,x,0)|^2dx.
 \]
 By Gronwall inequality, we have 
 \[
 \|v(t)\|_{L^2}\le \|v(0)\|_{L^2}+C_E\nu^{\frac{1-\beta}{2}}
 \]
 where $C_E$ is a constant that only depends on the Euler solution. The proof is complete.
 \end{proof}
 \begin{remark}One can see that the bound in the above theorem proves the inviscid limit precisely when $\beta<1$, as $\nu^{(1-\beta)/2}\to 0$ in the inviscid limit. Apparently, the proof fails to imply anything at the critical case $\beta=1$.\end{remark} 
 
In this paper, we give a direct proof of the inviscid limit for data with \textit{analytic regularity} in \textit{the critical case} $\beta=1$, with a precise pointwise bound on the vorticity in this class of initial data (see Section \ref{main} for the precise statement). The proof relies on our previous framework in \cite{2N}, which completely avoids boundary layer expansions. More precisely, we work with the vorticity formulation and the boundary conditions that capture \eqref{cri}, derive a pointwise bound for the Green function of the Stokes problems, and propagate boundary layer norms for the vorticity. The difficulty we have to overcome in our analysis of \eqref{cri} is the precise pointwise bound of the temporal Green function for the Stokes problems, which allows us to propagate the boundary layer norm in analytic function spaces.  Interestingly, we shall see below that the nonlinear iteration (with the Green kernel of Stokes) for the full Navier-Stokes equations with the boundary condition \eqref{cri} is \textit{just slightly} better than the no-slip boundary condition's iteration in \cite{2N}, due to a special cancellation of the pole in the resolvent analysis of the Green function (see Section \ref{secGreen}). This might support the intuition that if the fluid is allowed to slip even in the critical sense, it is less violent than the no-slip boundary condition (see \cite{Gre,ToanGrenier2}). 
Lastly, we remark that, just as for critical slip, the inviscid limit is also largely open for the classical no-slip boundary condition $u^\nu|_{\partial\Omega}=0$ (e.g, see \cite{SamCaf2,Mae,2N,1904.04983,Anna}).
~\\
~\\
{\bf Notations:} In this paper, for complex numbers $A,B$, we write $A\lesssim B$ to mean that $|A|\le C_0 |B|$ for some constant $C_0>0$ independent of viscosity $\nu>0$; we also denote $\Re A,\Im A$ to be the real and imaginary part of $A$ respectively. 
~\\~\\
{\bf Organization of the paper:} In section \ref{vor-formulation}, we derive a suitable boundary condition for the vorticity to ensure the critical slip boundary condition \eqref{cri}. In section \ref{sec-defBL}, we introduce analytic boundary layer norms for the vorticity.  In section \ref{main}, we state our main results. In section \ref{sec-analyticBL}, we recall several elliptic and bilinear estimates for the velocity and the nonlinear terms in analytic norm. In section \ref{sec-Stokes123}, we construct and derive a pointwise estimate for the Green function of the Stokes problem. We conclude the paper with Section \ref{sec-proof} with the proofs of the main theorems stated in Section \ref{main}. 
~\\
{\bf Acknowledgement:} The author would like to thank Toan T. Nguyen and Theodore D. Drivas for their many insightful discussions on the subject. The research was supported by the NSF under grant DMS-1764119. Part of this work was done while the author was visiting the Department of Mathematics and the Program in Applied and Computational Mathematics at Princeton University.

\section{Boundary vorticity formulation}\label{vor-formulation}

Let $\omega(x,z) = \partial_z u_1 - \partial_x u_2$ be the corresponding vorticity in $(x,z)\in \TT\times \RR_+$. Then, the vorticity equation reads
\beq\label{NS-vor}
\partial_{t}\omega-\nu\Delta\omega=-u\cdot\nabla\omega
\eeq
with $u = \nabla^\perp \Delta^{-1} \omega$. Here and throughout the paper, $\Delta^{-1}$ denotes the inverse of the Laplacian operator with the Dirichlet boundary condition: precisely, $\phi = \Delta^{-1}\omega$ solves $\Delta \phi = \omega$ on the half-space $\TT \times \RR_+$, with $\phi_{\vert_{z=0}} =0$.  

To ensure the critical slip boundary condition, we impose $\nu \w=u_1$ on the boundary. 
Taking Fourier transform in $x$, namely $\w(x,z)=\sum_{\alpha\in \mathbb{Z}}\w_\alpha(z)e^{i\alpha x}$, we impose the following boundary condition 
\beq\label{bdrcon}
\nu\w_\alpha(0)=-\int_0^\infty e^{-\alpha y}\w_\alpha(y)dy
\eeq 
which follows from the following lemma:

\begin{lemma}
Let $u_{1,\alpha}$ be the Fourier transform of the tangential component $u_1$ and $\w_\alpha$ the Fourier transform of $\w$. Then the value of $u_{1,\alpha}$ on the boundary $z=0$ is given by:
 \[
 u_{1,\alpha}(0)=-\int_0^\infty e^{-\alpha y}\w_\alpha(y)dy.
 \]
\end{lemma}

\begin{proof}
Since $\pt_xu_1+\pt_zu_2=0$, one can write $u_1=\pt_z\phi$ and $u_2=-\pt_x\phi$ for some stream function $\phi$.
Since $\triangle \phi=\w$ on $\mathbb{T}\times \mathbb{R}_+$ and $\phi|_{z=0}=0$, we have 
\[
(\pt_z^2-\alpha^2)\phi_\alpha=\w_\alpha,\qquad \phi_\alpha(0)=0
\]
where $\phi_\alpha$ and $\w_\alpha$ are the Fourier transform of $\phi$ and $\w$.
The solution of the above equation is given explicitly by 
\[\bega 
\phi_\alpha(z)&=\frac{1}{2\alpha}\int_0^\infty \lw(e^{-\alpha|y+z|}-e^{-\alpha|y-z|}\rw)\w_\alpha(y)dy\\
&=\frac{1}{2\alpha}\lw(
\int_0^\infty e^{-\alpha(y+z)}\w_\alpha(y)dy-\int_0^z e^{\alpha(y-z)}\w_\alpha(y)dy-\int_z^\infty e^{\alpha(z-y)}\w_\alpha(y)dy
\rw)
\enda 
\]
Since $u_{1,\alpha}=\pt_z\phi_\alpha$, a direct calculation yields 
\[
u_{1,\alpha}(z)=\frac{1}{2\alpha}\lw(-\alpha \int_0^\infty e^{-\alpha(y+z)}\w_\alpha(y)dy+\alpha \int_0^z e^{\alpha(y-z)}\w_\alpha(y)dy-\alpha \int_z^\infty e^{\alpha(z-y)}\w_\alpha(y)dy\rw)
\]
The lemma follows, after evaluating $u_{1,\alpha}$ at $z=0$.
\end{proof}

\section{Analytic boundary layer function spaces}\label{sec-defBL}

In this section, we recall analytic boundary layer spaces introduced in our previous work \cite{2N} (see also \cite{ToanGrenier1,ToanGrenier2}). Precisely, we consider holomorphic functions on the pencil-like complex domain:
\begin{equation}\label{def-pencil}\Omega_{\sigma} =\Big\{z\in\mathbb{C}:\quad|\Im z|<\min\{\sigma \Re z ,\sigma \}\Big\} ,\end{equation}
for $\sigma>0$. Let $\delta = \sqrt \nu$ be the classical boundary layer thickness. We introduce the analytic boundary layer function spaces ${\cal B}^{\sigma,\delta}$ that consists of holomorphic functions on $\Omega_\sigma$ with a finite norm 
\begin{equation}\label{def-bl0}
\| f \|_{\sigma,\delta}  = \sup_{z\in \Omega_\sigma} | f(z) | e^{\beta_0 \Re z} 
\Bigl( 1 + \delta^{-1} \phi_{P} (\delta^{-1} z)  \Bigr)^{-1}
\end{equation}
for some small $\beta_0>0$, and for boundary layer weight function $$ \phi_P(z) = \frac{1}{1+|\Re z|^P} $$
for some fixed constant $P>1$. Here, we suppress the dependence on $\beta_0,P$ as they are fixed throughout the paper. We expect that the vorticity function $\omega (t,x,z)$, for each fixed $t,x$, will be in ${\cal B}^{\sigma,\delta}$, precisely describing the behavior near the boundary and near infinity. In fact, there is an additional initial layer of thickness $\delta_t = \sqrt{\nu t}$ that appears near the boundary. To capture this, we introduce the time-dependent boundary layer norm: 
\begin{equation}\label{def-blt}
\| f\|_{\sigma,\delta(t)}  = \sup_{z\in \Omega_\sigma} | \omega(z) | e^{\beta_0 \Re z} 
\Bigl( 1 + \delta_t^{-1} \phi_{P} (\delta_t^{-1} z)  +  \delta^{-1} \phi_{P} (\delta^{-1} z)  \Bigr)^{-1} ,
\end{equation}
with $\delta_t = \sqrt{\nu t}$, $\delta = \sqrt \nu$, and with the same boundary layer weight function $\phi_P(\cdot)$. By convention, the norm $\| \cdot \|_{\sigma,\delta(0)}$ at time $t=0$ is replaced by $\| \cdot\|_{\sigma,\delta}$, the boundary layer norm with precisely one boundary layer behavior of thickness $\delta$, and $\| \cdot \|_{\sigma,0}$ denotes the norm without the boundary layer behavior.

For functions depending on two variables $f(x,z)$, we introduce the partial Fourier transform
in variable $x$ 
$$f(x,z)=\sum_{\alpha\in\mathbb{Z}}f_{\alpha}(z)e^{i\alpha x}$$
and introduce the following analytic norm 
$$||f|| _{\rho,\sigma,\delta(t)}=\sum_{\alpha\in\mathbb{Z}} e^{\rho|\alpha|} ||f_{\alpha}||_{\sigma,\delta(t)}$$
for $\rho,\sigma>0$. We denote by $B^{\rho,\sigma,\delta(t)}$ the corresponding spaces. 
In Section \ref{sec-analyticBL}, we shall recall some basic properties of such analytic function spaces.

\section{Main results}\label{main}

Our main results are as follows. 

\begin{theorem}\label{theo-main} Let $M_0>0$ and let $\omega_0$ be in ${\cal B}^{\rho_0,\sigma_0,\delta}$ for $\rho,\sigma>0$ and for $\delta = \sqrt \nu$, with $\|\omega_0 \|_{\rho_0,\sigma_0,\delta}\le M_0$. Then, there is a positive time $T$, independent of $\nu>0$,  so that the solution $\omega(t)$ to the Navier-Stokes equations \eqref{NS-vor}-\eqref{NS-BC}, with the initial data $\omega(0) = \omega_0$, exists in $C^1([0,T]; {\cal B}^{\rho,\sigma, \delta(t)})$ for $0<\rho <\rho_0$ and $0<\sigma<\sigma_0$. In particular, there is a $C_0$ so that the vorticity $\omega(t)$ satisfies 
\begin{equation}\label{bdry-propagate}|\omega(t,x,z)| \le C_0 e^{-\beta_0 z} \Bigl( 1 + \delta_t^{-1} \phi_{P} (\delta_t^{-1} z)  +  \delta^{-1} \phi_{P} (\delta^{-1} z)  \Bigr)\end{equation}
for $(t,x,z)\in [0,T]\times \TT \times \RR_+$, with $\delta_t = \sqrt{\nu t}$ and $\delta = \sqrt \nu$. 
\end{theorem}

\begin{theorem}\label{theo-limit}  Let $M_0>0$ and let $u_0^\nu$ be divergence-free analytic initial data so that  $\omega_0^\nu = \nabla \times u_0^\nu$ is in ${\cal B}^{\rho_0,\sigma_0,\delta}$ for $\rho,\sigma>0$ and for $\delta = \sqrt \nu$, with $\|\omega_0^\nu\|_{\rho_0,\sigma_0,\delta}\le M_0$. Then, the inviscid limit holds for Navier-Stokes solutions with the initial data $u^\nu_0$, with the time scale set by Theorem \ref{theo-main}. Precisely, there are unique local solutions $u^\nu(t)$ to the Navier-Stokes equations \eqref{NS}-\eqref{NS-BC}, for small $\nu>0$, and a unique solution $u^E(t)$ to the corresponding Euler equations, with initial data $u^E_0 = \lim_{\nu\to 0}u_0^\nu$, so that 
\[
\| u^\nu(t) - u^E(t)\|_{L^2}\le \|u^\nu_0-u_0^E\|_{L^2}+C_T\sqrt\nu+C_T(\nu t)^{\frac{1}{4}}\qquad\text{for}\quad t\in [0,T],
\]
where $C_T$ is a constant that only depends on the solution of Euler and $T$. In particular, we have 
\[
\sup_{0\le t\le T}\|u^\nu(t)-u^E(t)\|_{L^p}\to 0\qquad \text{as}\quad \nu\to 0
\]
for any $2\le p<\infty$. 
\end{theorem}

As mentioned, the proof of the main theorems is direct, using the vorticity formulation \eqref{NS-vor}-\eqref{bdrcon}. For Theorem \ref{theo-main}, we first prove the local existence of solutions in the analytic space $L^1_{\rho,\sigma}$ (see \eqref{def-Anorm}), and then in boundary layer spaces in order to establish the precise pointwise behavior of the vorticity (see Section \ref{bdrPropa}). Theorem \ref{theo-main} applies in particular for well-prepared analytic data that satisfy the Prandtl's ansatz of size $\sqrt \nu$. For general analytic data, beside the Prandtl's layers, the initial layers whose thickness is of order $\sqrt{\nu t}$ appear as captured in \eqref{bdry-propagate}. 
After proving Theorem \ref{theo-main}, we establish Theorem \ref{theo-limit} by a direct energy estimate; see Section \ref{proof}. 

\section{Analytic function spaces}\label{sec-analyticBL}
In this section, we recall basic properties of the analytic norms as well as the elliptic estimates that yield bounds on velocity in term of vorticity. These norms and estimates can be found in \cite{2N,SamCaf2}. 
Let $f(x,z)$ be holomorphic functions on $\TT \times \Omega_{\sigma}$, with $\Omega_\sigma$ being the pencil-like complex domain defined as in \eqref{def-pencil}. 
For $\rho,\sigma>0$ and $1\le p\le \infty$, we introduce the analytic function spaces denoted by $\mathcal L^p_{\rho,\sigma}$ with the finite norm 
\begin{equation}\label{def-Anorm} \| f\| _{\mathcal L^p_{\rho,\sigma}} :=\sum_{\alpha\in\mathbb{Z}} e^{\rho|\alpha|} \|f_{\alpha}\|_{L^p_\sigma}, \qquad \| f_\alpha \|_{L^p_\sigma} := \sup_{0\le \theta < \sigma}\Big( \int_{\partial\Omega_\theta} |f_\alpha(z)|^p\; |dz|\Big)^{1/p},\end{equation}
in which $f_\alpha = f_\alpha(z)$ denotes the Fourier transform of $f(x,z)$. In the case when $p=\infty$, the 
$L^p$ norm is replaced by the sup norm over $\Omega_\sigma$. Recalling the analytic boundary layer space $B^{\rho,\sigma,\delta(t)}$ introduced in Section \ref{sec-defBL}, we have

\begin{lemma}[$L^1$ embedding] \label{lem-emL1}There holds the embedding ${\cal B}^{\rho,\sigma,\delta(t)} \subset \mathcal L^1_{\rho,\sigma}$. 
\end{lemma}

\begin{lemma}[Recovering loss of derivatives] \label{loss-1}
For any $0<\sigma'<\sigma$, $0<\rho'<\rho$, and $\psi(z)=\frac{z}{1+z}$, there hold
\begin{equation}\label{alg1}\| fg\|  _{\mathcal L^1_{\rho,\sigma}}  \le\| f\| _{\mathcal L^\infty_{\rho,\sigma}}\| g\|_{\mathcal L^1_{\rho,\sigma}}, \end{equation}
\begin{equation}\label{alg2}\|\partial_{x}f\|_{\mathcal L^1_{\rho',\sigma}}\le \frac{C}{\rho- \rho'}\| f\|_{\mathcal L^1_{\rho,\sigma}}, \qquad \| \psi(z)\partial_{z}f\| _{\mathcal L^1_{\rho,\sigma'}}\le \frac{C}{\sigma-\sigma'}\| f\|_{\mathcal L^1_{\rho,\sigma}}. \end{equation}
The same estimates hold for boundary layer norms $\| \cdot \|_{\rho,\sigma,\delta}$ replacing $\|\cdot \|_{\mathcal L^1_{\rho,\sigma}}$ in the above three inequalities. 
\end{lemma}

\begin{lemma}[Elliptic estimates] \label{inverseLaplace}
Let $\phi$ be the solution of
$
- \Delta \phi = \omega
$
with the zero Dirichlet boundary condition, and set $u = \nabla^\perp \phi$. Then, there hold 
\beq \label{laplace-5}
\| u_1\|_{\mathcal L^\infty_{\rho,\sigma}}+\| u_2\| _{\mathcal L^\infty_{\rho,\sigma}}\le C\| \omega\| _{\mathcal L^1_{\rho,\sigma}} ,
\eeq 
\beq 
\label{laplace-6}
\| \partial_{x}u_1\| _{\mathcal L^\infty_{\rho,\sigma}}+\| \nabla u_2\| _{\mathcal L^\infty_{\rho,\sigma}} + \| \psi^{-1}u_2\| _{L^\infty_{\rho,\sigma}} \le C\| \omega\|  _{\mathcal L^1_{\rho,\sigma}}+C\|  \partial_{x}\omega\| _{\mathcal L^1_{\rho,\sigma}} ,
\eeq 
\beq \label{laplace-7}
\| \nabla u_1 \|_{\mathcal L^1_{\rho,\sigma}} + \|\nabla u_2 \|_{\mathcal L^1_{\rho,\sigma}} \le C\| \omega\|  _{\mathcal L^1_{\rho,\sigma}}, 
\eeq 
 with $\psi(z) = z/ (1+z)$, for some constant $C$.
\end{lemma}

\begin{lemma}[Bilinear estimates]\label{lem-bilinear}
For any $\omega$ and $\tilde \omega$,
denoting by $v$ the velocity related to $\omega$, we have
$$
\begin{aligned}
 \| v\cdot \nabla \tilde \omega \|_{\mathcal L^1_{\rho,\sigma}} &\le C \| \omega\|_{\mathcal L^1_{\rho,\sigma}}\| \tilde \omega_x\|_{\mathcal L^1_{\rho,\sigma}} + C (\| \omega\|_{\mathcal L^1_{\rho,\sigma}} +\|\omega_x\|_{\mathcal L^1_{\rho,\sigma}} ) \| \psi(z)\partial_z\tilde \omega\|_{\mathcal L^1_{\rho,\sigma}}
 \\
  \| v\cdot \nabla \tilde \omega \|_{\rho,\sigma,\delta} &\le C \| \omega\|_{\rho,\sigma,\delta}\| \tilde \omega_x\|_{\rho,\sigma,\delta} + C (\| \omega\|_{\rho,\sigma,\delta} +\|\omega_x\|_{\rho,\sigma,\delta} ) \| \psi(z)\partial_z\tilde \omega\|_{\rho,\sigma,\delta}
 \end{aligned}$$
\end{lemma}

~\\
We refer the readers to (\cite {2N}, Section 2) for detailed proofs of the above lemmas. 

\section{The Stokes problem}\label{sec-Stokes123}

\subsection{Main propositions}
In this section, we state our Proposition \ref{prop-Stokes} for the inhomogenous Stokes problem 
\beq 
\begin{cases}\label{stokes}
\omega_{t}-\nu\triangle\omega&=f(t,x,y), \qquad \mbox{in}\quad \TT \times \Omega_\sigma,\\
\nu \omega &=u_1, \quad\qquad \mbox{on} \quad y=0,\\
u_2|_{y=0}&=0,\\
\w|_{t=0}&=\w_0.\\
\end{cases}
\eeq 
Let $e^{\nu t B}$ denote the semigroup of the corresponding Stokes problem: namely, the heat equation $\partial_{t}\omega-\nu\Delta\omega = 0$ on $\TT \times \Omega_\sigma$ with the homogenous boundary condition $
(\nu\w-u_1)|_{y=0}=0$.
Solutions to the linear Stokes problem is then constructed via the following Duhamel's integral representation:
 \beq\label{Duh-w}
\omega(t)=e^{\nu tB}\omega_{0} +\int_{0}^{t}e^{\nu(t-s)B} f(s) \; ds
\eeq 

In this section, we shall derive uniform bounds for the Stokes semigroup in analytic spaces, with the analytic norm 
$$\| \omega\|  _{\rho,\sigma,\delta(t)}=\sum_{\alpha\in\mathbb{Z}} e^{\rho|\alpha|} \|\omega_{\alpha}\|_{\sigma,\delta(t)}$$
with the boundary layer norm defined by 
\begin{equation}\label{def-blnormFull}
\| \omega_\alpha\|_{\sigma,\delta(t)}  = \sup_{z\in \Omega_{\sigma}} | \omega_\alpha(z) | e^{\beta \Re z} 
\Bigl( 1 + \delta_t^{-1} \phi_{P} (\delta_t^{-1} z)  +  \delta^{-1} \phi_{P} (\delta^{-1} z)  \Bigr)^{-1} ,
\end{equation}
in which the boundary thicknesses are $\delta_t = \sqrt{\nu t}$ and $\delta = \sqrt \nu $.  As for the initial data, the norm is measured by $\| \omega_\alpha\|_{\sigma,\delta(0)}$, which consists of precisely one boundary layer behavior whose thickness is $\delta = \sqrt \nu$. We introduce 
$$ ||| \omega(t) |||_{\rho,\sigma,\delta(t),k}  = \sum_{j+\ell \le k} \|\partial_x^j (\psi(z)\partial_z)^\ell \omega(t)\|_{\rho,\sigma,\delta(t)}$$
and 
$$ ||| \omega |||_{\mathcal{W}^{k,1}_{\rho,\sigma}} = \sum_{j+\ell \le k} \|\partial_x^j (\psi(z)\partial_z)^\ell \omega(t)\|_{\mathcal L^1_{\rho,\sigma}} .$$ 
 Next, we state our main proposition, which will be proved in Section \ref{conv-proof}: \begin{proposition}\label{prop-Stokes} Let $e^{\nu t B}$ be the semigroup for the linear Stokes problem. Then, $\partial_x$ commutes with $e^{\nu t B}$. In addition, for any $k\ge 0$, and for any $0\le s< t\le T$, there hold 
$$
\begin{aligned}
 ||| e^{\nu t B} f  |||_{\rho,\sigma,\delta(t),k} & \lesssim  ||| f|||_{\rho,\sigma, \delta(0),k},
  \\ ||| e^{\nu (t-s)B}f  |||_{\rho,\sigma,\delta(t),k} & \lesssim \sqrt{\frac t{t-s}}|||f|||_{\mathcal{W}^{k,1}_{\rho,\sigma}}+\sqrt{\frac{t}{s}} ||| f|||_{\rho,\sigma,\delta(s),k} ,
   \end{aligned}$$
uniformly in the inviscid limit. Similarly, we also obtain 
$$
\begin{aligned}
 ||| e^{\nu t B} f  |||_{\mathcal{W}^{k,1}_{\rho,\sigma}}  \lesssim  ||| f|||_{\mathcal{W}^{k,1}_{\rho,\sigma}},    \end{aligned}$$
uniformly in the inviscid limit. 
\end{proposition}

\subsection{Duhamel principle}
We first treat the Stokes problem on $\TT \times \RR_+$. By  taking the Fourier transform in $x$, the problem is reduced to 
\beq\label{Stokes-a} 
\begin{aligned} 
\partial_t\omega_\alpha-\nu \Delta_\alpha\omega_\alpha&=f_\alpha(t,z), \qquad \mbox{in}\quad \RR_+\\
\nu \w_\alpha(0)&=-\int_0^\infty e^{-\alpha y}\w_\alpha(y)dy\\
\end{aligned}
\eeq 
in which $\omega_\alpha$ denotes the Fourier transform of $\omega$ with respect to $x$, and $\Delta_\alpha = \partial_z^2 - \alpha^2$. Let $G_\alpha(t,z,y)$ be the corresponding Green function of the linear Stokes problem \eqref{Stokes-a}, together with the initial data $G_\alpha(0,z,y) = \delta_y(z)$.  For each fixed $y \ge 0$, the function $G_\alpha(t,z,y)$ solves 
\beq\label{Stokes-Gr} 
\begin{aligned} 
( \partial_t-\nu \Delta_\alpha) G_\alpha(t,z,y) &=0, \qquad \mbox{in}\quad \RR_+,\\
\nu G_\alpha(t,0,y)&=-\int_0^\infty e^{-\alpha z}G_\alpha(t,z,y)dz
\end{aligned}
\eeq 
 together with the initial data $G_\alpha(0,z,y) = \delta_y(z)$.  
 \begin{proposition}
The solution to \eqref{Stokes-a} is constructed via Duhamel's principle: 
 \begin{equation}\label{Duh-Stokes}
\begin{aligned}
 \omega_\alpha(t,z) &= \int_0^\infty  G_\alpha(t,z,y) \omega_{0,\alpha} (y) \; dy+ \int_0^t\int_0^\infty G_\alpha(t-s,z,y) f_\alpha (s,y)\; dyds. 
\end{aligned} 
\end{equation}
\end{proposition}
\begin{proof}
We first show that $\pt_t\w_\alpha-\nu\triangle_{\alpha}\w_\alpha=f_\alpha$. Without loss of generality, we can assume $\w_{0,\alpha}(y)=0$.
We have 
\[\bega
\pt_t\w_\alpha&=\frac{d}{dt}\lw(\int_0^t \int_0^\infty G_\alpha(t-s,z,y)f_\alpha(s,y)dyds\rw)\\
&=\int_0^\infty G_\alpha(0,z,y)f_\alpha(t,y)dy+\int_0^t \int_0^\infty \lw(\pt_t G_\alpha(t-s,z,y)\rw)f_\alpha(y,s)dyds\\
&=f_\alpha(t,z)+\nu\int_0^t \int_0^\infty\triangle_{\alpha} G_\alpha(t-s,z,y)f_\alpha(y,s)dyds\\
&=f_\alpha(t,z)+\nu\triangle_\alpha \w_\alpha.
\enda
\]
We now check the boundary condition in \eqref{Stokes-a}. Let $z=0$, we have 
\[\bega
\nu \w_\alpha(t,0)&=\nu\int_0^\infty G_\alpha(t,0,y)\w_{0,\alpha}(y)dy+\nu\int_0^t \int_0^\infty G_\alpha(t-s,0,y)f_\alpha(s,y)dyds\\
&=-\int_0^\infty \int_0^\infty \lw(e^{-\alpha z}G_\alpha(t,0,y)dz\rw)\w_{0,\alpha}(y)dy-\int_0^t \int_0^\infty \lw(\int_0^\infty e^{-\alpha z}G_\alpha(t-s,z,y)dz\rw)f_\alpha(s,y)dyds\\
&=-\int_0^\infty e^{-\alpha z}\w_\alpha(t,z)dz.\\
\enda
\]
The proof is complete.
 \end{proof}

 \subsection{The Green function for the Stokes problem}\label{secGreen}
In this section, we derive sufficient pointwise bounds on the temporal Green function for the linear Stokes problem \eqref{Stokes-a}. Precisely, we prove the following. 

\begin{proposition}\label{prop-Stokes-Green} Let $G_{\alpha}(t,z,y)$ be the Green function of the Stokes problem \eqref{Stokes-a}. There holds 
\begin{equation}\label{Stokes-tG}
G_\alpha(t,z,y) = H_\alpha(t,z,y) + R_\alpha (t,z,y), 
\end{equation}
in which $H_\alpha(t,z,y) $ is exactly the one-dimensional heat kernel with the homogenous Neumann boundary condition and $R_\alpha(t,z,y)$ is the residual kernel due to the boundary condition. Precisely, 
There hold 
$$
\begin{aligned}
H_\alpha(t,z,y) & = \frac{1}{\sqrt{4\pi\nu t}} \Big( e^{-\frac{|y-z|^{2}}{4\nu t}} +  e^{-\frac{|y+z|^{2}}{4\nu t}} \Big) e^{-\alpha^{2}\nu t}, 
\\ | \partial_z^k R_\alpha (t,z,y)| &\lesssim (\nu t)^{-k/2}e^{-\theta_0\alpha^2\nu t}\cdot (\nu t)^{-1/2}e^{-\theta_0\frac{ z^2}{4\nu t}}
\end{aligned}$$
for $y,z\ge 0$, $k\ge 0$, and for some $\theta_0>0$. 
\end{proposition}
We proceed the construction of the Green function via the resolvent equation. Namely, for each fixed $y\ge 0$, let $G_{\lambda,\alpha}(y,z)$ be the $L^1$ solution to the resolvent problem 
\beq\label{Stokes-res}
\bega 
( \lambda - \nu \Delta_\alpha )G_{\lambda,\alpha}(y,z)&=\delta_{y}(z)\\
\nu G_{\lambda,\alpha}(0,y) &=-\int_0^\infty e^{-\alpha z}G_{\lambda,\alpha}(z,y)dz.
\enda 
\eeq 
Here, the second non-local boundary condition is derived as follows:
Given a forcing term  $f(y)$, we look for solution of the form:
\[
\w_{\lambda,\alpha}(z)=\int_0^\infty G_{\lambda,\alpha}(z,y)f(y)dy
\]
We define the operator $L$ to be $L=-\nu(\pt_z^2-\mu^2)$, where $\mu=\sqrt{\frac{\lambda}{\nu}+\alpha^2}$ with positive real part. Then we get 
\[
L\w_{\lambda,\alpha}(z)=\int_0^\infty LG_{\lambda,\alpha}(y,z)f(y)dy=\int_0^\infty \delta(y-z)f(y)dy=f(z).
\]
Putting this in the boundary condition \eqref{Stokes-a} we get
\[\bega
\nu \int_0^\infty G_{\lambda,\alpha}(0,y)f(y)dy&=-\int_0^\infty e^{-\alpha z}\lw(\int_0^\infty G_{\lambda,\alpha}(z,y)f(y)dy\rw)dz.\\
\enda\]
Hence we have 
\[\bega
\int_0^\infty(\nu G_{\lambda,\alpha}(0,y))f(y)dy&=-\int_0^\infty \lw(\int_0^\infty e^{-\alpha z}G_{\lambda,\alpha}(z,y)dz\rw)f(y)dy.
\enda 
\]
Thus we take the following condition on the Green function 
\beq \label{bdr1}
\nu G_{\lambda,\alpha}(0,y)=-\int_0^\infty e^{-\alpha z}G_{\lambda,\alpha}(z,y)dz\qquad \text{for any}\qquad y\ge 0
\eeq 
We then obtain the following:

\begin{lemma} Let $\mu = \nu^{-1/2}\sqrt{\lambda + \alpha^2 \nu}$, having positive real part. There holds 
\beq \label{green}
G_{\lambda,\alpha}(z,y) = H_{\lambda,\alpha}(z,y) + R_{\lambda,\alpha}(z,y)
\eeq 
in which $H_{\lambda,\alpha}(y,z) $ denotes the resolvent kernel of the heat problem with homogenous Neumann boundary condition and $R_{\alpha,\lambda}(z,y)$ denotes the residue resolvent kernel; namely, 
\[
\begin{cases}
H_{\lambda,\alpha}(z,y)&=\frac{1}{2\mu \nu}\lw(e^{-\mu|y-z|}+e^{-\mu(y+z)}\rw),\\
R_{\lambda,\alpha}(z,y)&=\frac{1}{\mu \nu}\frac{\alpha-\lambda}{\lambda+\mu-\alpha}e^{-\mu(y+z)}-\frac{1}{\nu(\lambda+\mu-\alpha)}e^{-\alpha y-\mu z}.\\
\end{cases}
\]
In particular, $G_{\lambda,\alpha}(y,z)$ is meromorphic with respect to $\lambda$ in $\CC \setminus \{ -\alpha^2 \nu - \RR_+\}$ with a pole at $\lambda =0$. 
\end{lemma}
\begin{proof} 
We have 
\beq \label{form1}
G_{\lambda,\alpha}(z,y)=\begin{cases}
c_1(y) e^{\mu z}+c_2(y) e^{-\mu z}&,\qquad z<y\\
c_3(y) e^{-\mu z}&,\qquad z>y\\
\end{cases}
\eeq 
The continuity of $G_{\lambda,\alpha}$ at $z=y$ gives 
\beq \label{eq1}
c_1(y)e^{2\mu y}+c_2(y)=c_3(y).
\eeq
Now, the jump condition of $-\nu \pt_zG_{\lambda,\alpha}$ at $z=y$ gives 
\beq \label{eq2}
c_3(y)=-c_1(y)e^{2\mu y}+c_2(y)+\frac{1}{\mu \nu}e^{\mu y}.\\
\eeq
Combining \eqref{eq1} and \eqref{eq2}, we get 
\beq \label{c1}
c_1(y)=\frac{1}{2\mu \nu}e^{-\mu y}.
\eeq
and hence \eqref{eq2} becomes 
\beq\label{eq3}
c_3(y)=c_2(y)+\frac{1}{2\mu \nu}e^{\mu y}.
\eeq
Now we find $c_2$. Using the boundary condition \eqref{bdr1} and the form of $G_{\lambda,\alpha}$ in \eqref{form1}, we have 
\[
-\nu (c_1(y)+c_2(y))=\int_0^ye^{-\alpha z}\lw(c_1(y)e^{\mu z}+c_2(y)e^{-\mu z}\rw)dz+\int_y^\infty e^{-\alpha z}\lw(c_3(y)e^{-\mu z}\rw)dz.\]
By a direct calculation, we get
\beq\label{c2}
c_2(y)=\frac{1}{2\mu \nu}\frac{\mu+\alpha-\lambda}{\lambda+\mu-\alpha}e^{-\mu y}-\frac{1}{\nu(\lambda+\mu-\alpha)}e^{-\alpha y}.
\eeq
Combining the above equation with \eqref{eq3}, we have 
\beq\label{c3}
c_3(y)=\frac{1}{2\mu\nu}e^{\mu y}+\frac{1}{2\mu \nu}\frac{\mu+\alpha-\lambda}{\lambda+\mu-\alpha}e^{-\mu y}-\frac{1}{\nu(\lambda+\mu-\alpha)}e^{-\alpha y}.
\eeq
Hence, putting $c_1,c_2,c_3$, computed in \eqref{c1},\eqref{c2},\eqref{c3}, in the formula of $G_{\lambda,\alpha}(z,y)$ in \eqref{form1}, we get
\[\bega
G_\lambda(z,y)
&=H_{\lambda,\alpha}(z,y)+R_{\lambda,\alpha}(z,y),\\
\enda 
\]
where 
\[
\begin{cases}
H_{\lambda,\alpha}(z,y)&=\frac{1}{2\mu \nu}\lw(e^{-\mu|y-z|}+e^{-\mu(y+z)}\rw)\\
R_{\lambda,\alpha}(z,y)&=\frac{1}{\mu \nu}\frac{\alpha-\lambda}{\lambda+\mu-\alpha}e^{-\mu(y+z)}-\frac{1}{\nu(\lambda+\mu-\alpha)}e^{-\alpha y-\mu z}.\\
\end{cases}
\]
This completes the proof. \end{proof}

\begin{proof}[Proof of Proposition \ref{prop-Stokes-Green}]

The temporal Green function $G_\alpha(t,z,y)$ can then be constructed via the inverse Laplace transform: 
\begin{equation}\label{Stokes-Green} G_{\alpha}(t,z,y)=\frac{1}{2\pi i}\int_\Gamma e^{\lambda t}G_{\lambda,\alpha}(z,y)d\lambda\end{equation}
in which the contour of integration $\Gamma$ is taken such that it remains on the right of the (say, $L^2$) spectrum of the linear operator $\lambda - \nu \Delta_\alpha$, which is $-\alpha^2 \nu - \RR_+$. 

In view of \eqref{green}, we set $H_{\alpha}(t,z,y)$ and $R_{\alpha}(t,z,y)$ to be the corresponding temporal Green function of $H_{\lambda,\alpha}(z,y)$
 and $R_{\lambda,\alpha}(z,y)$, respectively. It follows that $H_{\alpha}(t,z,y)$ is the temporal Green function of the one-dimensional heat problem with the homogenous Neumann boundary condition, yielding 
 $$H_{\alpha}(t,z,y) = \frac{1}{\sqrt{4\pi \nu t}} \Big( e^{-\frac{|y-z|^2}{4\nu t}} + e^{-\frac{|y+ z|^2}{4\nu t}} \Big) e^{-\nu \alpha^2 t}.$$
~\\
It remains to compute the residual Green function
\beq\label{Rbound}
\bega
R_\alpha(t,z,y)&=\dfrac{1}{2\pi i}\int_\Gamma e^{\lambda t}R_{\lambda,\alpha}(z,y)d\lambda,\\
R_{\lambda,\alpha}(z,y)&=\frac{1}{\mu \nu}\frac{\alpha-\lambda}{\lambda+\mu-\alpha}e^{-\mu(y+z)}-\frac{1}{\nu(\lambda+\mu-\alpha)}e^{-\alpha y-\mu z}.\\
\enda
\eeq
We note that $R_{\lambda,\alpha}$ has a pole when $\lambda+\mu-\alpha=0$, which happens only when $\lambda=0$.\\
We consider two cases: when $\alpha^2\nu\le 1$ and when $\alpha^2\nu\ge 1$.\\

~\\
\textbf{Case 1:} $\alpha^2\nu\le 1$. \\
Let us give a bound on the first part of the kernel $R_{\lambda,\alpha}$ in \eqref{Rbound}:
\[
R_{\lambda,\alpha}^1(z,y)=\frac{1}{\mu \nu}\cdot\frac{\alpha-\lambda}{\lambda+\mu-\alpha}e^{-\mu(y+z)}.
\]
By Cauchy's theory, we may decompose the contour of integration as $\Gamma=\Gamma_\pm\cup\Gamma_c$, having 
\[
\bega 
\Gamma_\pm&=\bigg\{\lambda=-\frac{1}{2}\alpha^2\nu+\nu(a^2-b^2)+2ab\nu i\pm iM,\qquad \pm b\in \mathbb{R}_+\bigg\},
\\
\Gamma_c&=\bigg\{\lambda=-\frac{1}{2}\alpha^2\nu+\nu a^2+Me^{i\theta},\quad \theta\in [-\pi/2,\pi/2]\bigg\}.
\enda 
\]
for some positive number $M$ and $a=\frac{|y+z|}{2\nu t}$. Since $\alpha^2\nu\le 1$, we can  take $M$ large so that the pole $\lambda=0$ remains on the left of the contour $\Gamma$. It is clear that $|\lambda |\gtrsim 1$ on $\Gamma$. 
 ~\\
On $\Gamma_c$, we note that 
$$
\begin{aligned}
\Re \mu &=  \nu^{-1/2} \Re \sqrt{\frac12 \nu \alpha^2 + \nu a^2 + M e^{i\theta}} \ge \nu^{-1/2} \sqrt{\frac12 \nu \alpha^2 + \nu a^2}  \ge a,
\\
\Re \mu &=  \nu^{-1/2} \Re \sqrt{\frac12 \nu \alpha^2 + \nu a^2 + M e^{i\theta}} \ge c_0\nu^{-1/2} \sqrt{M}
\end{aligned}$$ for some $c_0>0$. 
~\\
This implies that $\Re \mu \ge \frac a2 +\frac a2$ and  $|\mu|\nu \ge c_0 \nu^{1/2} $.  This proves that 
$$ \begin{aligned}
&\Big |\int_{\Gamma_c} e^{\lambda t} e^{-\mu(y+z)} \lw(\frac{1}{\mu \nu}\frac{\alpha-\lambda}{\lambda+\mu-\alpha}\rw)\; d\lambda\Big| \\
&\lesssim \int_{-\pi/2}^{\pi/2} e^{Mt-\frac{1}{2}\alpha^2\nu t} e^{ a^2 \nu t} e^{ - \frac a2 |y+z|} e^{-\frac{a}{2}|y+z| } \nu^{-1/2}  d\theta  \cdot \sup_{\lambda\in \Gamma_c}\lw|\frac{\alpha-\lambda}{\lambda+\mu-\alpha}\rw|
\\
&\lesssim \nu^{-1/2}e^{-\frac a2|y+z| } e^{ a^2 \nu t} e^{ - \frac a2 |y+z|} e^{-\frac{1}{2}\alpha^2\nu t}
\\
&\lesssim \nu^{-1/2} e^{-\frac a2|y+z| } e^{-\frac{1}{2}\alpha^2\nu t}\\
&\lesssim (\nu t)^{-1/2} e^{-\frac {|y+z|^2 }{4\nu t}}e^{-\frac{1}{2}\alpha^2\nu t}\\
\end{aligned}$$ 
in which we used $e^{ a^2 \nu t} e^{ - \frac a2 |y+z|}   = 1$ by definition of $a$, and the fact that $\lw|\frac{\alpha-\lambda}{\lambda+\mu-\alpha}\rw|$ is bounded on $\Gamma_c$. Indeed, we write
\[\bega 
\lw|\frac{\alpha-\lambda}{\lambda+\mu-\alpha}\rw|&= \lw|-1+\frac{\mu}{\lambda+\mu-\alpha}\rw|\le1+ \lw|\dfrac{\mu}{\lambda+\mu-\alpha}\rw|.\\
\enda \]
It suffices to estimate $\lw|\frac{\mu}{\lambda+\mu-\alpha}\rw|$ when $\lambda\in \Gamma_c$. Using the fact that $\lambda=\nu(\mu^2-\alpha^2)$, we can rewrite this term as follows:
\beq\label{quant}
\lw(1+\frac{\alpha}{\mu-\alpha}\rw)\frac{1}{\nu(\mu+\alpha)+1}.
\eeq
First we see that $\frac{\alpha}{\mu-\alpha}$ is bounded, since 
\beq\label{quant1}
|\mu-\alpha|\ge \Re\mu-\alpha \ge c_0\sqrt M\nu^{-1/2}-\alpha\ge c_0\sqrt M\alpha-\alpha =(c_0\sqrt M-1)\alpha\qquad (\text{since}\quad \alpha^2\nu \le 1).
\eeq
Moreover, we get
\beq\label{quant2}
|\nu(\mu+\alpha)+1|\ge 1+\alpha \nu+\nu\Re \mu \ge 1.
\eeq
Hence the quantity \eqref{quant} is uniformly bounded when $\lambda\in \Gamma_c$. This implies that 

\[
\sup_{\lambda\in \Gamma_c}\lw|\frac{\alpha-\lambda}{\lambda+\mu-\alpha}\rw|\lesssim 1
\] as claimed.\\
Now we estimate the term
\beq\label{term1}
\int_{\Gamma_\pm}e^{\lambda t}\frac{1}{\mu \nu}\cdot \frac{\alpha-\lambda}{\lambda+\mu-\alpha}e^{-\mu(y+z)}d\lambda.
\eeq
On $\Gamma_\pm$, we note that  
$$
\begin{aligned}\Re\mu &= \Re \sqrt{\frac12 \alpha^2 + (a+ib)^2 \pm i \nu^{-1} M} \ge \Re \sqrt{(a+ib)^2} = a,\end{aligned}$$
upon noting that the sign of $b$ and $\pm M$ is the same on $\Gamma_\pm$. Similarly, we note that $\Re \mu \gtrsim \sqrt M/\sqrt \nu$. 
By definition of $a$, we have 
 $$ | e^{\lambda t} e^{-\mu|y+z|} |\le  e^{-\frac12\nu \alpha^2 t} e^{-\frac{|y+z|^2}{4\nu t}} e^{-\nu b^2 t},$$ 
Moreover, by a similar argument as in \eqref{quant},\eqref{quant1} and \eqref{quant2}, we get
\[
\sup_{\lambda\in \Gamma_\pm}\lw|\frac{\alpha-\lambda}{\lambda+\mu-\alpha}\rw|
\lesssim 1.
\]
Thus we get the following bound for the term in \eqref{term1} as follows:
\[
\lw|\int_{\Gamma_\pm} e^{\lambda t}\frac{1}{\mu \nu}\frac{\alpha-\lambda}{\lambda+\mu-\alpha}e^{-\mu(y+z)}d\lambda\rw|\lesssim (\nu t)^{-1/2}e^{-\frac{1}{2}\alpha^2\nu t}e^{-\frac{|y+z|^2}{4\nu t}}.
\]
The proof of the bound for $\int_{\Gamma}e^{\lambda t}R_{\lambda,\alpha}^1(y,z)d\lambda$ is complete.
 Similarly, we get the following bound for the second term in the kernel \eqref{Rbound}
 \[
 \lw|\frac{1}{2\pi i}\int_\Gamma e^{\lambda t}\frac{1}{\nu(\lambda+\alpha-\mu)}e^{-\alpha y-\mu z}d\lambda\rw|\lesssim e^{-\alpha y}(\nu t)^{-1/2}e^{-\frac{1}{2}\alpha^2\nu t}e^{-\frac{z^2}{4\nu t}}, \]
 which we skip the details.
 This completes the proof the case $\alpha^2\nu \le 1$.\\

~\\
\textbf{Case 2:} $\alpha^2\nu \ge 1$.\\
Take $a=\frac{z}{2\nu t} $. Consider first the case when $|a-\alpha|\ge \frac12\alpha$. In this case, we move the contour of integration to  
$$\Gamma_1 := \Big\{ \lambda  =- \nu \alpha^2 + \nu (a^2 - b^2 ) + 2\nu ia b, \quad \pm b \in \RR_+\Big\} $$
which may pass the pole at $\lambda =0$ (precisely, it does when $a =\alpha$). By the Cauchy's theory, we have $$R_{\alpha}(t,z,y)=\frac{1}{2\pi i}\int_{\Gamma_1} e^{\lambda t} R_{\lambda,\alpha}(z,y) \; d\lambda  + \mathrm{Res}_0$$
in which the residue at the pole $\lambda =0$ is computed explicitly by \begin{equation}\label{residue} \mathrm{Res}_0 = 0. \end{equation}
Indeed, at the pole $\lambda=0$, we have $\mu=\alpha$. Hence 
\[
(\lambda+\mu-\alpha)R_{\lambda,\alpha}=\frac{\alpha }{\mu \nu}e^{-\mu(y+z)}-\frac{1}{\nu}e^{-\alpha y-\mu z}=0,\qquad \text{since} \qquad \mu=\alpha.
\]
Hence, we have 
\[\bega
R_\alpha(t,z,y)&=\frac{1}{2\pi i}\int_{\Gamma_1}e^{\lambda t}R_{\lambda,\alpha}(y,z)d\lambda\\
\enda
\]
where 
\[\begin{cases}
R_{\lambda,\alpha}^1(z,y)&=\frac{1}{\mu\nu}\cdot \frac{\alpha-\lambda}{\lambda+\mu-\alpha}e^{-\mu(y+z)},\\
R_{\lambda,\alpha}^2(z,y)&=-\frac{1}{\nu(\lambda+\mu-\alpha)}e^{-\alpha y-\mu z}.\\
\end{cases}
\]
Now we estimate 
 \beq\label{est}
 \bega
\lw |\frac{1}{2\pi i}\int_{\Gamma_1}e^{\lambda t}R^1_{\lambda,\alpha}(z,y)d\lambda \rw|&\lesssim \int_{\Gamma_1}e^{\Re\lambda t}\frac{1}{\nu|\mu|}\lw|\frac{\alpha-\lambda}{\lambda+\mu-\alpha}\rw|e^{-\Re\mu(y+z)}|d\lambda|
\\
 &\lesssim \sup_{\lambda\in \Gamma_1}\lw|\frac{\alpha-\lambda}{\lambda+\mu-\alpha}\rw|\int_{\mathbb{R}}e^{-\alpha^2\nu t+\nu a^2 t-\nu b^2 t}\lw(e^{-\frac{a}{2}z}e^{-\frac{a }{2}z}\rw)db \\
 &\lesssim (\nu t)^{-1/2} e^{-\frac{z^2}{4\nu t}}e^{-\alpha^2\nu t}.
 \enda
 \eeq
  Here, we used the fact that $e^{\nu a^2t}e^{-\frac{a}{2}|y+z|}=1, |d\lambda|=\nu|d\mu|$ and 
 \beq\label{claim}
 \sup_{\lambda\in \Gamma_1}\lw|\frac{\alpha-\lambda}{\lambda+\mu-\alpha}\rw|\lesssim 1.
  \eeq
  Indeed, we have 
\beq\label{est1}
\bega 
& \lw|\frac{\alpha-\lambda}{\lambda+\mu-\alpha}\rw|=\lw|-1+\frac{\mu}{\lambda+\mu-\alpha}\rw|\le 1+\lw|\frac{\mu}{\lambda+\mu-\alpha}\rw|=1+\lw|\frac{\mu}{(\mu-\alpha)(\nu(\mu+\alpha)+1)}\rw|\\
&\le 1+\lw|\lw(1+\frac{\alpha}{\mu-\alpha}\rw)\frac{1}{\nu(\mu+\alpha)+1}
\rw|\le 1+\lw(1+\frac{\alpha}{|\mu-\alpha|}\rw)\frac{1}{|\nu\mu+\alpha\nu+1|}\\
&\le 1+\lw(1+\frac{\alpha}{|\mu-\alpha|}\rw)\lesssim 1,
\enda 
\eeq
since $|\mu-\alpha|\ge |\Re\mu-\alpha|=|a-\alpha|\ge \frac{1}{2}\alpha$. 
The bound for $\lw|\int_{\Gamma_1}e^{\lambda t}R_{\lambda,\alpha}^1(z,y)d\lambda\rw|$ is complete.
Similarly, one can obtain the following bound 
\[
\lw|
\frac{1}{2\pi i}\int_{\Gamma_1}e^{\lambda t}R_{\lambda,\alpha}^2(z,y)d\lambda
\rw|\lesssim (\nu t)^{-1/2}e^{-\frac{z^2}{4\nu t}}e^{-\alpha^2\nu t}e^{-\alpha y},
\]
which we skip the details.
Combining the above bounds for $R_{\lambda,\alpha}^1$ and $R_{\lambda,\alpha}^2$, we have 
 \[
 R_\alpha(t,z,y)\lesssim (\nu t)^{-1/2}e^{-\alpha^2\nu t}e^{-\frac{z^2}{4\nu t}}.
 \]
 It remains to consider the case when $|a-\alpha| \le \frac 12\alpha$ and $\alpha^2 \nu \ge 1$. We note in particular that $\frac 12 \alpha \le a\le \frac 32\alpha$. In this case, we take the contour of integration as follows
 $$\Gamma_2 := \Big\{ \lambda  =- \frac18\nu \alpha^2 + \nu (a^2 - b^2 ) + 2\nu ia b, \quad \pm b \in \RR_+\Big\} .$$
Observe that the contour $\Gamma_1$ always leaves the origin on the left, hence the pole at the origin does not appear. Proceeding as in the estimate \eqref{est} and \eqref{est1}, it suffices to check that 
\beq\label{key}
\sup_{\lambda\in \Gamma_2}\lw|\frac{\alpha}{\mu-\alpha}\rw|\lesssim 1
\eeq
in order to conclude 
\beq\label{con1}
\lw|
\frac{1}{2\pi i}\int_{\Gamma_2} e^{\lambda t} R_{\lambda,\alpha}^1(z,y)d\lambda 
\rw|\lesssim(\nu t)^{-1/2}e^{-\frac{z^2}{4\nu t}}e^{-\frac{1}{8}\alpha^2\nu}.
\eeq
To check \eqref{key}, we first see that the contour $\Gamma_2$ cuts the real axis at $\nu\lw(a^2-\frac{1}{8}\alpha^2\rw)$ and cuts the imaginary axis at $\pm 2a\nu\sqrt{a^2-\frac{1}{8}\alpha^2}$. In particular this implies
\[
|\lambda|\ge \nu\lw(a^2-\frac{1}{8}\alpha^2\rw)\ge\nu \lw(\frac{1}{4}\alpha^2-\frac{1}{8}\alpha^2\rw)\ge \frac{1}{8}\alpha^2\nu,\qquad \text{since}\quad a\ge \frac{1}{2}\alpha.
\]
Hence we have 
\beq\label{lambda-bound}
|\lambda|\ge \frac{1}{8}\alpha^2\nu.
\eeq
Now using the fact $\lambda=\nu(\mu^2-\alpha^2)$ and \eqref{lambda-bound}, we see that
\[
\lw|\frac{\alpha}{\mu-\alpha}\rw|=\lw|\frac{\alpha\nu(\mu+\alpha)}{\nu(\mu^2-\alpha^2)}\rw|=\frac{|\alpha^2\nu+\alpha\nu\mu|}{|\lambda|}\le \frac{\alpha^2\nu}{|\lambda|}+\frac{\alpha \nu|\mu|}{|\lambda|}\le 8+\frac{\alpha\nu|\mu|}{|\lambda|}.
\]
Now to bound $\frac{\alpha \nu |\mu|}{|\lambda|}$, we note that $\lambda=\nu(\mu^2-\alpha^2)$ and \eqref{lambda-bound}, and hence
  \[
 \nu |\mu|^2\le |\lambda|+\alpha^2\nu\le 9|\lambda|.
 \]
 Thus \[
 \frac{\alpha\nu|\mu|}{|\lambda|}\lesssim \frac{\alpha \nu |\mu|}{\nu|\mu|^2}=\frac{\alpha}{|\mu|}\lesssim\frac{\alpha}{\Re\mu}\lesssim  \frac{\alpha}{a}\lesssim 1.
 \]
 This completes the proof of the bound stated in \eqref{con1}.
As for the derivatives bound, it is straight forward that 
\[
|\pt_z^k H_\alpha(t,z,y)|\lesssim (\nu t)^{-\frac{k}{2}} \frac{1}{\sqrt{\nu t}}\lw(e^{-\theta_0\frac{|y-z|^2}{4\nu t}}+e^{-\theta_0\frac{|y+z|^2}{4\nu t}}\rw),\qquad k\ge 1
\]
for some $\theta_0>0$. For the residue kernel $R_\alpha(t,z,y)=\frac{1}{2\pi i}\int_{\Gamma}e^{\lambda t}R_{\lambda,\alpha}(z,y)d\lambda$, we note that 
\[
\pt_z\lw(\frac{1}{2\pi i}\int_{\Gamma}e^{\lambda t}R_{\lambda,\alpha}(z,y)d\lambda\rw)=\frac{1}{2\pi i}\int_{\Gamma}e^{\lambda t}\mu R_\alpha(z,y)d\lambda 
\]
Hence, we get 
\[
|\pt_z R_\alpha(t,z,y)|\lesssim (\nu t)^{-1/2}\cdot \frac{1}{\sqrt{\nu t}}e^{-\theta_0\frac{z^2}{4\nu t}}e^{-\theta_0\alpha^2\nu t}
\]
by the exact same argument represented for the bound $R_\alpha(t,z,y)$, and the fact that $\int_{\mathbb R}be^{-\nu t b^2}db\lesssim (\nu t)^{-1/2}$ and $\frac{z}{\nu t}e^{-\frac{z^2}{4\nu t}}\lesssim (\nu t)^{-1/2}e^{-\theta_0\frac{z^2}{4\nu t}}$, which we skip the details (see also \cite{2N}).
\end{proof}

\subsection{The Green function on $\Omega_{\sigma}$}\label{sec-Grcomplex}
The Green function constructed in Proposition \ref{prop-Stokes-Green} can be directly extended to the complex domain $\Omega_{\sigma}$ defined by 
$$
\Omega_{\sigma}= \Big\{z\in\mathbb{C}:\quad|\Im z|<\min\{\sigma|\Re z|,\sigma \}\Big\},
$$
for some small $\sigma>0$. Indeed, the Green function involves precisely the heat kernel $G(t,z) = \frac{1}{\sqrt{4\pi t}} e^{-z^2 / 4t}$, which is extended to the complex domain. In addition, we note that for $z\in \Gamma_{\sigma}$, there holds $\Im z \le \sigma \Re z$, which implies that 
$$ |e^{-z^2 / 4t}| \le e^{- |\Re z|^2/4t + |\Im z|^2 / 4t} \le e^{- (1-\sigma^2) |\Re z|^2 / 4t}.$$
Similar estimates hold for the other terms in the Green function $G_{\alpha}(t,z,y) =  H_\alpha(t,z,y) + R_\alpha (t,z,y)$, yielding \begin{equation}\label{Stokes-Gr-complex}
\begin{aligned}
H_\alpha(t,z,y) &\lesssim \frac{1}{\sqrt{\nu t}} \Big( e^{-(1-\sigma^2)\frac{|\Re y- \Re z|^{2}}{4\nu t}} +  e^{-(1-\sigma^2)\frac{|\Re y+\Re z|^{2}}{4\nu t}} \Big) e^{-\frac{1}{8}\alpha^{2}\nu t}, 
\\R_\alpha (t,z,y) &\lesssim e^{-\theta_0 \alpha^2\nu t} (\nu t)^{-1/2}e^{-\theta_0(1-\sigma^2) \frac{(\Re z)^2}{4\nu t}},
\end{aligned}\end{equation}
for $y,z \in \Gamma_{\sigma}$, and for some $\theta_0>0$. Precisely, for any $z\in \Omega_\sigma$, let $\theta$ be the positive constant so that $z\in \partial \Omega_\theta$. The Duhamel principle \eqref{Duh-Stokes} then becomes 
 \begin{equation}\label{Duh-Stokes-a}
\begin{aligned}
 \omega_\alpha(t,z) &= \int_{\partial\Omega_\theta}G_\alpha(t,z,y) \omega_{0,\alpha} (y) \; dy + \int_0^t  \int_{\partial\Omega_\theta} G_\alpha(t-s,z,y) f_\alpha (s,y)\; dyds,
\end{aligned} \end{equation}
which is well-defined for $z\in \Omega_\sigma$, having the Green function $G_\alpha(t,z,y) $ satisfies the pointwise estimates \eqref{Stokes-Gr-complex}, similar to those on the real line. For this reason, it suffices to derive convolution estimates for real values $y,z$.

\subsection{Convolution estimates}\label{sec-convS}

We now derive convolution estimates. We start with the analytic $L^1$ norms. For $k\ge 0$, we introduce   
$$ \| \omega_\alpha\|_{\mathcal{W}^{k,1}_{\sigma}} = \sum_{j=0}^k \|(\psi(z)\partial_z)^j \omega_\alpha \|_{L^1_\sigma} .$$ 
We prove the following.

\begin{proposition}\label{prop-Stokes-convL1} Let $T>0$ and let $G_{\alpha}(t,z,y)$ be the Green function of the Stokes problem \eqref{Stokes-a}, constructed in Proposition \ref{prop-Stokes-Green}. 
Then, for any $0\le s <  t\le T$ and $k\ge 0$, there is a universal constant $C_T$ so that 
$$
\begin{aligned}
\Big \| \int_0^\infty G_\alpha(t,\cdot,y) \omega_\alpha(y)\; dy \Big\|_{\mathcal{W}^{k,1}_{\sigma}} &\le C_T   \| \omega_\alpha\|_{\mathcal{W}^{k,1}_{\sigma}},
\\
\Big \| \int_0^\infty G_\alpha(t-s,\cdot,y) \omega_\alpha(y,s)\; dy \Big\|_{\mathcal{W}^{k,1}_{\sigma}} &\le C_T  \| \omega_\alpha(s)\|_{\mathcal{W}^{k,1}_{\sigma}},
\end{aligned}$$
uniformly in the inviscid limit.  
\end{proposition}
\begin{proof} We shall prove the convolution for real values $y,z$. For the complex extension, see Section \ref{sec-Grcomplex}. 
Recall from  Proposition \ref{prop-Stokes-Green} that $
G_\alpha(t,z,y) = H_\alpha(t,z,y) + R_\alpha (t,z,y), 
$
with $$
\begin{cases}
H_\alpha(t,z,y) & = \frac{1}{\sqrt{4\pi\nu t}} \Big( e^{-\frac{|y-z|^{2}}{4\nu t}} +  e^{-\frac{|y+z|^{2}}{4\nu t}} \Big) e^{-\alpha^{2}\nu t}, 
\\ R_\alpha (t,z,y) &\lesssim e^{-\theta_0\alpha^2\nu t}(\nu t)^{-\frac{1}{2}}e^{-\theta_0 \frac{z^2}{\nu t}}.
\end{cases}$$
For $H_\alpha$, we apply (\cite{2N}, Proposition 3.7) to get: 
\[
\Big \| \int_0^\infty H_\alpha(t-s,\cdot,y) \omega_\alpha(y,s)\; dy \Big\|_{\mathcal{W}^{k,1}_{\sigma}} \le C_T  \| \omega_\alpha(s)\|_{\mathcal{W}^{k,1}_{\sigma}}.
\]
Now we will prove that  
\[
\Big \| \int_0^\infty R_\alpha(t-s,\cdot,y) \omega_\alpha(y,s)\; dy \Big\|_{\mathcal{W}^{k,1}_{\sigma}} \le C_T  \| \omega_\alpha(s)\|_{\mathcal{W}^{k,1}_{\sigma}}.
\]
Using the pointwise bound of $R_\alpha(t-s,z,y)$ in Proposition \ref{prop-Stokes-Green}, we have 
\[
\lw|\int_0^\infty R_\alpha(t-s,z,y)\w_\alpha(s,y)dy
\rw|\lesssim e^{-\theta_0 \alpha^2\nu(t-s)}e^{-\theta_0 \frac{z^2}{4\nu (t-s)}}(\nu(t-s))^{-1/2}\int_0^\infty|\w_\alpha(s,y)|dy.
\]
Integrating in $z$, we have 
\[
\lw\|\int_0^\infty R_\alpha(t-s,z,y)\w_\alpha(s,y)dy
\rw\|_{L^1_z}\lesssim \|\w_\alpha(s)\|_{L^1_y}.
\]
As for derivatives, we have 
\beq\label{est}
\bega 
\lw|(\psi(z)\pt_z)^k\lw(\int_0^\infty R_\alpha(t-s,y,z)\w_\alpha(s,y)dy\rw)\rw|&\lesssim \lw(\frac{z^2}{\nu (t-s)}\rw)^k(\nu (t-s))^{-1/2} e^{-\theta_0\frac{z^2}{\nu (t-s)}}\int_0^\infty |\w_\alpha(s,y)|dy\\
&\lesssim (\nu (t-s))^{-1/2} e^{-\theta_0\frac{z^2}{\nu (t-s)}}\|\w_\alpha(s)\|_{L^1_y}.
\enda 
\eeq 
From here, we get 
\[
\lw \|(\psi(z)\pt_z)^k\int_0^\infty R_\alpha(t-s,y,z)\w_\alpha(s,y)dy\rw\|_{L^1_z}\lesssim \|\w_\alpha(s)\|_{L^1_y}.
\]
The proof is complete.
\end{proof}

\subsection{Convolution estimates with boundary layer behaviors}\label{conv-proof}
In this section, we provide the convolution estimates of the Green function against functions in the boundary layer spaces, whose norm is defined by 
\begin{equation}\label{def-blnormFull123}
\| \omega_\alpha\|_{\sigma,\delta(t)}  = \sup_{z\in \Omega_\sigma} | \omega_\alpha(z) | e^{\beta \Re z}
\Bigl( 1 + \delta_t^{-1} \phi_{P} (\delta_t^{-1} z)  +  \delta^{-1} \phi_{P} (\delta^{-1} z)  \Bigr)^{-1} ,
\end{equation}
for $t>0$ and $\beta>0$, in which the boundary thicknesses are $\delta_t = \sqrt{\nu t}$ and $\delta = \sqrt \nu $ and for boundary layer weight $\phi_P(z) = \frac{1}{1+|\Re z|^P}$, $P>1$. We also introduce the boundary norm for derivatives: 
$$ \| \omega_\alpha\|_{\sigma,\delta(t),k} = \sum_{j=0}^k \|(\psi(z)\partial_z)^j \omega_\alpha \|_{\sigma,\delta(t)} $$ 
for $k\ge 0$. In the case $t=0$, the norm $\|\cdot \|_{\sigma,\delta(0)}$ is defined to consist of precisely one boundary layer with thickness $\delta = \sqrt \nu$.

We prove the following.

\begin{proposition}\label{prop-Stokes-conv} Let $T>0$ and let $G_{\alpha}(t,z,y)$ be the Green function of the Stokes problem \eqref{Stokes-a}, constructed in Proposition \ref{prop-Stokes-Green}. 
Then, for any $0\le s <  t\le T$ and $k\ge 0$, there is a universal constant $C_T$ so that 
$$
\begin{aligned}
\Big \| \int_0^\infty G_\alpha(t,\cdot,y) \omega_\alpha(y)\; dy \Big\|_{\sigma, \delta(t),k} &\le C_T   \| \omega_\alpha\|_{\sigma, \delta(0),k},
\\
\Big \| \int_0^\infty G_\alpha(t-s,\cdot,y) \omega_\alpha(s,y)\; dy \Big\|_{\sigma, \delta(t),k} &\le C_T \sqrt{\frac ts}  \| \omega_\alpha(s)\|_{\sigma, \delta(s),k}+C_T\sqrt{\frac{t}{t-s}}\|\w_\alpha(s)\|_{\mathcal{W}^{k,1}_\sigma}
\end{aligned}$$
uniformly in the inviscid limit.  
\end{proposition}
\begin{proof}
Since $G_\alpha(t-s,z,y)=H_\alpha(t-s,z,y)+R_\alpha(t-s,z,y)$, the convolution estimates are needed for the heat kernel $H_\alpha$ and $R_\alpha$. For $H_\alpha$, we apply (\cite{2N}, Lemma 3.10) to get   
\[
\Big \| \int_0^\infty H_\alpha(t-s,\cdot,y) \omega_\alpha(s,y)\; dy \Big\|_{\sigma, \delta(t),k} \le C_T \sqrt{\frac ts}  \| \omega_\alpha(s)\|_{\sigma, \delta(s),k}.
\]
Now we will prove that 
\[
\Big \| \int_0^\infty R_\alpha(t-s,\cdot,y) \omega_\alpha(s,y)\; dy \Big\|_{\sigma, \delta(t),k} \le C_T \sqrt{\frac {t}{t-s}}  \| \omega_\alpha(s)\|_{\sigma, \delta(s),k}.
\]
By the estimate \eqref{est}, it suffices to check that 
\[
(\nu(t-s))^{-1/2}e^{-\theta_0\alpha^2\nu(t-s)}e^{-\theta_0\frac{z^2}{4\nu(t-s)}}\lesssim \sqrt{\frac{t}{t-s}}e^{-\beta_0 z}\lw(\delta_t^{-1}\phi_P(\delta_t^{-1}z)\rw).
\]
To this end, we have 
\[\bega
(\nu(t-s))^{-1/2}e^{-\theta_0\alpha^2\nu(t-s)}e^{-\theta_0\frac{z^2}{4\nu(t-s)}}
&=\sqrt{\frac{t}{t-s}}\delta_t^{-1}e^{-\theta_0 \frac{z^2}{8\nu(t-s)}}e^{-\theta_0\frac{z^2}{8\nu(t-s)}}e^{-\theta_0\alpha^2\nu(t-s)}\\
&\lesssim \sqrt{\frac{t}{t-s}}\lw(\delta_t^{-1} e^{-\theta_0\frac{z^2}{8\nu t}}\rw) e^{-\theta_0\frac{z^2}{8\nu(t-s)}}e^{-32\cdot \theta_0 \nu(t-s)}e^{32\cdot\theta_0\nu(t-s)}\\
&\lesssim \sqrt{\frac{t}{t-s}}\lw(\delta_t^{-1}\phi_P(\delta_t^{-1}z)\rw)e^{-\beta_0 z}
\enda 
\]
as long as $ \beta_0\le 2 \theta_0$, by a simple Cauchy inequality $\frac{z^2}{8\nu(t-s)}+32\nu(t-s)\ge 2z$.
The proof is complete.
\end{proof}

\section{Proof of the main theorems}\label{sec-proof}
As mentioned in the introduction, we construct the solutions to the Navier-Stokes equation via the vorticity formulation: 
\beq\label{NS-vor1}
\begin{aligned}
\partial_{t}\omega-\nu\Delta\omega &=-u\cdot\nabla\omega
\\
(\nu \w-u_1)\vert_{z=0}&=0, 
\end{aligned}\eeq
in which $u = \nabla^\perp \Delta^{-1} \omega$, with $\Delta^{-1}$ being the inverse of Laplacian with the Dirichlet boundary condition. For convenience, we set $N = u \cdot \nabla \omega$.  
The solution to the Navier-Stokes is then constructed via the Duhamel's principle:
 \beq\label{Duh-w1}
\omega(t)=e^{\nu tB}\omega_{0} -\int_{0}^{t}e^{\nu(t-s)B}  N(s) \; ds\eeq 
with $\omega_0\in {\cal B}^{\rho_0,\sigma_0,\delta}$, for some $\rho_0,\sigma_0>0$. 

\subsection{Nonlinear iteration}\label{sec-nonlinear}

Let us fix positive numbers $\gamma, \zeta,$ and $\rho_0$, and introduce the following nonlinear iterative norm for vorticity: 
\beq\label{def-normw}
\bega 
A(\gamma)=&\quad \sup_{0<\gamma t< \rho_0}\sup_{\rho<\rho_0- \gamma t}\Bigl{\{} 
 ||| \omega(t)|||_{\mathcal{W}^{1,1}_{\rho,\rho}} +  ||| \omega(t)|||_{\mathcal{W}^{2,1}_{\rho,\rho}}(\rho_0-\rho-\gamma t)^{\zeta}\Bigr{\}}\\
\enda 
\eeq 
with recalling $$ ||| \omega(t)|||_{\mathcal{W}^{k,1}_{\rho,\rho}}  = \sum_{j+\ell \le k}  \|\partial_x^j (\psi(z)\partial_z)^\ell \omega(t)\|_{L^1_{\rho,\rho}}.$$
Here, for sake of presentation, we take the same analyticity radius in $x$ and $z$; namely, $\sigma = \rho <\rho_0$. Thanks to Lemma \ref{lem-emL1}, $\omega_0 \in \mathcal{W}^{k,1}_{\rho,\rho}$, for any $k\ge 0$. 

We shall show that the vorticity norm remains finite for sufficiently large $\gamma$. The weight $(\rho_0-\rho-\gamma t)^\zeta$, with a small $\zeta>0$, is standard to avoid time singularity when recovering the loss of derivatives (\cite{caflisch1990,Safonov}). Let $\rho < \rho_0 - \gamma t$. Thanks to Lemma \ref{lem-bilinear}, we have \begin{equation}\label{non-est}
\begin{aligned}
 ||| N(t)|||_{\mathcal{W}^{0,1}_{\rho,\rho}} &\lesssim  ||| \omega(t)|||_{\mathcal{W}^{1,1}_{\rho,\rho}}^2 \le A(\gamma)^2,
 \\
 ||| N(t)|||_{\mathcal{W}^{1,1}_{\rho,\rho}} &\lesssim  ||| \omega(t)|||_{\mathcal{W}^{1,1}_{\rho,\rho}}  ||| \omega(t)|||_{\mathcal{W}^{2,1}_{\rho,\rho}}  \le A(\gamma)^2 (\rho_0-\rho-\gamma t)^{-\zeta}
. \end{aligned}\end{equation}

Now, using the Duhamel integral formula \eqref{Duh-w1}, we estimate 
$$
\begin{aligned}
||| \omega(t) |||_{\mathcal{W}^{k,1}_{\rho,\rho}} \le ||| e^{\nu tB}\omega_{0}|||_{\mathcal{W}^{k,1}_{\rho,\rho}} +\int_{0}^{t} ||| e^{\nu(t-s)B} N (s)|||_{\mathcal{W}^{k,1}_{\rho,\rho}} \; ds .
\end{aligned}
$$
In view of Proposition \ref{prop-Stokes}, the term from the initial data is already estimated, giving $||| e^{\nu tB}\omega_{0}|||_{\mathcal{W}^{k,1}_{\rho,\rho}} \le \|\omega_0\|_{\mathcal{W}^{k,1}_{\rho,\rho}}$. As for the integral terms, we estimate 
$$\begin{aligned}
\int_{0}^{t} ||| e^{\nu(t-s)B} N(s)|||_{\mathcal{W}^{1,1}_{\rho,\rho}} \; ds 
&\le C_0 \int_0^t ||| N(s)|||_{\mathcal{W}^{1,1}_{\rho,\rho}} \; ds
\\&\le C_0 A(\gamma)^2 \int_0^t  (\rho_0-\rho-\gamma s)^{-\zeta}\; ds
\\&\le 
C_0 \gamma^{-1} A(\gamma)^2 .
\end{aligned}$$
 Next, we give estimates for $k=2$. Noting that $\rho < \rho_0 - \gamma t \le \rho_0 - \gamma s$, we take $\rho' = \frac{\rho + \rho_0 - \gamma s}{2}$ and compute 
$$\begin{aligned}
\int_{0}^{t} ||| e^{\nu(t-s)B} N(s)|||_{\mathcal{W}^{2,1}_{\rho,\rho}} \; ds 
&\le C_0 \int_0^t ||| N(s)|||_{\mathcal{W}^{2,1}_{\rho,\rho}} \; ds 
\\&\le C_0 \int_0^t\frac{1}{\rho' - \rho}||| N(s)|||_{\mathcal{W}^{1,1}_{\rho',\rho'}}\; ds
\\&\le C_0 A(\gamma)^2 \int_0^t  (\rho_0-\rho-\gamma s)^{-1-\zeta}\; ds
\\&\le C_0 \gamma^{-1} A(\gamma)^2 (\rho_0-\rho-\gamma t)^{-\zeta}.
\end{aligned}$$
Same computation holds for the trace operator $\Gamma(\nu t)$, yielding 
$$ A(\gamma ) \le  C_0 \|\omega_0\|_{\mathcal{W}^{2,1}_{\rho,\rho}}  + C_0 \gamma^{-1} A(\gamma)^2 . $$ 
By taking $\gamma$ sufficiently large, the above yields the uniform bound on the iterative norm in term of initial data. This yields the local solution in $L^1_{\rho,\rho}$ for $t \in [0,T]$, with $T = \gamma^{-1}\rho_0$. 

\subsection{Propagation of boundary layers}\label{bdrPropa}
It remains to prove that the constructed solution has the boundary layer behavior as expected, having already constructed solutions in $L^1_{\rho,\rho}$ spaces. Indeed, we now introduce the following nonlinear iterative norm for vorticity: 
\beq\label{def-normw}
\bega 
B(\gamma)=&\quad \sup_{0<\gamma t< \rho_0}\sup_{\rho<\rho_0- \gamma t}\Bigl{\{} 
 ||| \omega(t)|||_{\rho,\delta(t),1} +  ||| \omega(t)|||_{\rho,\delta(t),2}(\rho_0-\rho-\gamma t)^{\zeta}\Bigr{\}}\\
\enda 
\eeq 
with the boundary layer norm $$ ||| \omega(t)|||_{\rho,\delta(t),k}  = \sum_{j+\ell \le k} \|\partial_x^j (\psi(z)\partial_z)^\ell \omega(t)\|_{\rho,\rho,\delta(t)} .$$ Thanks to Lemma \ref{lem-bilinear}, we estimate 
\begin{equation}\label{non-est}
\begin{aligned}
 ||| N(t)|||_{\rho,\delta(t),0} &\lesssim  ||| \omega(t)|||_{\rho,\delta(t),1}^2 \le B(\gamma)^2
 \\
 ||| N(t)|||_{\rho,\delta(t),1} &\lesssim  ||| \omega(t)|||_{\rho,\delta(t),1}  ||| \omega(t)|||_{\rho,\delta(t),2}  \le B(\gamma)^2 (\rho_0-\rho-\gamma t)^{-\zeta}
. \end{aligned}\end{equation}
Now, using the Duhamel integral formula \eqref{Duh-w1}, we estimate 
$$
\begin{aligned}
||| \omega(t) |||_{\rho,\delta(t),k} \le ||| e^{\nu tB}\omega_{0}|||_{\rho,\delta(t),k} +\int_{0}^{t} ||| e^{\nu(t-s)B} N (s)|||_{\rho,\delta(t),k} \; ds 
\end{aligned}
$$
In view of Proposition \ref{prop-Stokes}, the term from the initial data is already estimated, giving $||| e^{\nu tB}\omega_{0}|||_{\rho,\delta(t),k} \le \|\omega_0\|_{\rho,\delta(0),k}$. We estimate 
$$\begin{aligned}
&\int_{0}^{t} ||| e^{\nu(t-s)B} N(s)|||_{\rho,\delta(t),1} \; ds 
\\&\lesssim \int_0^t \lw(\sqrt{\frac{t}{s}}||| N(s)|||_{\rho,\delta(s),1}+\sqrt{\frac{t}{t-s}}|||N(s)|||_{\mathcal{W}^{1,1}_{\rho,\rho}}\rw) \; ds\\
&\lesssim B(\gamma)^2 \int_0^t\sqrt{\frac{t}{s}}  (\rho_0-\rho-\gamma s)^{-\zeta}ds+\sup_{0\le s\le T}|||N(s)|||_{\mathcal{W}^{1,1}_{\rho,\rho}}\int_0^t\sqrt{\frac{t}{t-s}}\; ds
\\
&\lesssim B(\gamma)^2 \Big( \int_0^{t/2}+ \int_{t/2}^t \Big) \sqrt{\frac{t}{s}}  (\rho_0-\rho-\gamma s)^{-\zeta}\; ds+t\cdot \sup_{0\le s\le T}|||N(s)|||_{\mathcal{W}^{1,1}_{\rho,\rho}}
\\
&\le C_0 B(\gamma)^2 \Big( t (\rho_0 - \rho - \frac12 \gamma t)^{-\zeta} + \frac1\gamma (\rho_0-\rho-\frac12\gamma t)^{1-\zeta}\Big)+t\cdot \sup_{0\le s\le T}|||N(s)|||_{\mathcal{W}^{1,1}_{\rho,\rho}}\\
&\le C_0 \gamma^{-1} B(\gamma)^2 (\rho_0 - \rho)^{-\zeta}+t\cdot \sup_{0\le s\le T}|||N(s)|||_{\mathcal{W}^{1,1}_{\rho,\rho}},
\end{aligned}$$
in which we used $\gamma t\le \rho_0$ and $\gamma t < \rho_0 - \rho$. Next, noting that  $\rho < \rho_0 - \gamma t \le \rho_0 - \gamma s$, we take $\rho' = \frac{\rho + \rho_0 - \gamma s}{2}$ and compute 
$$\begin{aligned}
&\int_{0}^{t} ||| e^{\nu(t-s)B} N(s)|||_{\rho,\delta(t),2} \; ds 
\\&\lesssim \int_0^t \lw(\sqrt{\frac{t}{s}}||| N(s)|||_{\rho,\delta(s),2}+\sqrt{\frac{t}{t-s}}|||N(s)|||_{\mathcal{W}^{2,1}_{\rho,\rho}}\rw) \; ds \\
&\lesssim  \int_0^t \sqrt{\frac{t}{s}} \frac{1}{\rho' - \rho}||| N(s)|||_{\rho',\delta(s),1}\; ds+t\cdot\sup_{0\le s\le T}|||N(s)|||_{\mathcal{W}^{2,1}_{\rho,\rho}}
\\&\lesssim  B(\gamma)^2 \int_0^t \sqrt{\frac{t}{s}}  (\rho_0-\rho-\gamma s)^{-1-\zeta}\; ds+t\cdot\sup_{0\le s\le T}|||N(s)|||_{\mathcal{W}^{2,1}_{\rho,\rho}}
\\
&\le C_0 B(\gamma)^2 \Big( \int_0^{t/2}+ \int_{t/2}^t \Big) \sqrt{\frac{t}{s}}  (\rho_0-\rho-\gamma s)^{-1-\zeta}\; ds +t\cdot\sup_{0\le s\le T}|||N(s)|||_{\mathcal{W}^{2,1}_{\rho,\rho}}\\
&\le C_0 B(\gamma)^2 \Big( t (\rho_0 - \rho - \frac12 \gamma t)^{-1-\zeta} + \frac1\gamma (\rho_0-\rho-\gamma t)^{-\zeta}\Big)+t\cdot\sup_{0\le s\le T}|||N(s)|||_{\mathcal{W}^{2,1}_{\rho,\rho}}\\
&\le C_0 \gamma^{-1} B(\gamma)^2 (\rho_0-\rho-\gamma t)^{-\zeta}+t\cdot\sup_{0\le s\le T}|||N(s)|||_{\mathcal{W}^{2,1}_{\rho,\rho}}.\end{aligned}$$
This proves the boundedness of the iterative norm $B(\gamma)$, and hence the propagation of the boundary layer behaviors. Theorem \ref{theo-main} follows. 

\subsection{Proof  of the inviscid limit}\label{proof}
In this section, we conclude the paper by proving the inviscid limit of Navier-Stokes for the critical slip boundary condition \eqref{cri}.
~\\
~\\
\textit{Proof of theorem \ref{theo-limit}}. Let $u^E\in W^{2,\infty}(\Omega)\cap W^{2,2}(\Omega)$ be the solution to Euler (in our case, $u^E$ is even analytic). As in \eqref{Sva2}, we have
\beq\label{ineq1}
\bega 
&\frac{1}{2}\frac{d}{dt}\|v\|_{L^2}^2+\int_\Omega (v\cdot \nabla u^E)\cdot v+\nu \int_\Omega \nabla u^E\cdot \nabla v+\nu\int_{\Omega}|\nabla v|^2\\
&+ \int_{\mathbb{T}}|u_1^\nu(t,x,0)|^2dx-\nu\int_{\mathbb{T}}\w^\nu(t,x,0)u_1^E(t,x,0)=0.\\
\enda 
\eeq
By Cauchy inequality, we have 
\[
\dfrac{d}{dt}\|v\|_{L^2}^2\lesssim C_E\lw(\|v\|_{L^2}^2+\nu+\nu\int_{\mathbb{T}}|\w^\nu(t,x,0)|dx\rw),
\]
where $C_E$ is a constant only depending on $u^E$. Now, since $\|\w^\nu(t)\|_{\sigma,\rho,\delta(t)}$ is uniformly bounded in $\nu$, there exists $C_0>0$ such that 
\[
|\w^\nu(t,x,y)|\le C_0 e^{-\beta_0 y}\lw(1+\delta^{-1}\phi_P(\delta^{-1}y)+\delta_t^{-1}\phi_P(\delta_t^{-1}y)\rw).
\]
Putting $y=0$, we get 
\beq \label{ineq3}
|\w^\nu(t,x,0)|\lesssim \delta_t^{-1}.
\eeq 
Combining \eqref{ineq1} and \eqref{ineq3}, we get 
\[
\dfrac{d}{dt}\|v(t)\|_{L^2}^2\lesssim \|v(t)\|_{L^2}^2+\dfrac{\sqrt{\nu}}{\sqrt{t}}+\nu.
\]
Hence, by Gronwall inequality, we get
\[
\|v(t)\|_{L^2}\lesssim (\nu t)^{1/4}+\sqrt\nu+\|v(0)\|_{L^2}.
\]
The proof is complete.

\bibliographystyle{abbrv}

\def\cprime{$'$} \def\cprime{$'$}

%

\end{document}